\documentclass{article}
\usepackage{graphicx} 

\usepackage[noadjust]{cite}
\usepackage{xcolor}
\usepackage{url}
\usepackage{latexsym,epsfig,amsfonts}
\usepackage{authblk}
\usepackage{tikz} 
\usepackage{verbatim}

\usepackage{amsthm}
\usepackage{amsmath}

\RequirePackage[all]{xy}

\newtheorem{thm}{Theorem}[section]
\newtheorem{cor1}[thm]{Corollary}

\newtheorem{prop1}[thm]{Proposition}

\newtheorem{obs1}[thm]{Observation}
\theoremstyle{definition}

\theoremstyle{remark}

\numberwithin{equation}{section}

\newcommand{\BibTeX}{B\kern-0.1emi\kern-0.017emb\kern-0.15em\TeX}
\newcommand{\XYpic}{$\mathrm{X\kern-0.3em\raisebox{-0.18em}{Y}}$-$\mathrm{pic}\,$}

\newcommand{\cl}{C \kern -0.1em \ell}  

\newcommand{\ed}{\end{document}}

\begin{document}

%
%
%
%
%
%
%
%
%

\title{The palette index of the Cartesian product of paths,  cycles and regular graphs}

\author{Aleksander Vesel}
\affil{%
Faculty of Natural Sciences and Mathematics\\
University of Maribor\\
Koro\v ska cesta 160\\
 SI-2000 Maribor, Slovenia and \\
IMFM\\
Jadranska 19\\
 SI-1000 Ljubljana, Slovenia}

\date{\today}

\makeatletter
\def\tikz@foreach{%
  \def\pgffor@beginhook{%
    \tikz@lastx=\tikz@foreach@save@lastx%
    \tikz@lasty=\tikz@foreach@save@lasty%
    \tikz@lastxsaved=\tikz@foreach@save@lastxsaved%
    \tikz@lastysaved=\tikz@foreach@save@lastysaved%
    \let\tikz@moveto@waiting=\tikz@foreach@moveto@waiting
    \setbox\tikz@figbox=\box\tikz@tempbox\expandafter\tikz@scan@next@command\pgfutil@firstofone}%
  \def\pgffor@endhook{\pgfextra{%
      \xdef\tikz@foreach@save@lastx{\the\tikz@lastx}%
      \xdef\tikz@foreach@save@lasty{\the\tikz@lasty}%
      \xdef\tikz@foreach@save@lastxsaved{\the\tikz@lastxsaved}%
      \xdef\tikz@foreach@save@lastysaved{\the\tikz@lastysaved}%
      \global\let\tikz@foreach@moveto@waiting=\tikz@moveto@waiting
      \global\setbox\tikz@tempbox=\box\tikz@figbox\pgfutil@gobble}}%
  \def\pgffor@afterhook{%
    \tikz@lastx=\tikz@foreach@save@lastx%
    \tikz@lasty=\tikz@foreach@save@lasty%
    \tikz@lastxsaved=\tikz@foreach@save@lastxsaved%
    \tikz@lastysaved=\tikz@foreach@save@lastysaved%
    \let\tikz@moveto@waiting=\tikz@foreach@moveto@waiting
    \setbox\tikz@figbox=\box\tikz@tempbox\tikz@scan@next@command}%
  \global\setbox\tikz@tempbox=\box\tikz@figbox%
  \xdef\tikz@foreach@save@lastx{\the\tikz@lastx}%
  \xdef\tikz@foreach@save@lasty{\the\tikz@lasty}%
  \xdef\tikz@foreach@save@lastxsaved{\the\tikz@lastxsaved}%
  \xdef\tikz@foreach@save@lastysaved{\the\tikz@lastysaved}%
    \global\let\tikz@foreach@moveto@waiting=\tikz@moveto@waiting
  \foreach}
\makeatother

\maketitle

\begin{abstract}
The palette of a vertex \( v \) in a graph \( G \) is the set of colors assigned to the edges incident to \( v \).  
The palette index of \( G \) is the minimum number of distinct palettes among the vertices, taken over all proper edge colorings of \( G \).  
This paper presents results on the palette index of the Cartesian product \( G \Box H \), where one of the factor graphs is a path or a cycle.  
Additionally, it provides exact results and bounds on the palette index of the Cartesian product of two graphs, where one factor graph is isomorphic to a regular or class 1 nearly regular graph.  
\end{abstract} 

\

\section{Introduction}

Let $G=(V, E)$ be a simple connected graph.
An {\em edge-coloring} of $G$ is a map that assigns colors to the edges of $G$.
An edge coloring is {\em proper} if two incident edges obtain different colors.
A {\em $k$-edge-coloring}  of $G$ is a proper edge-coloring with colors from the set 
$\{1,\ldots, k\}$. The minimum number of colors required in a proper edge-coloring of a graph $G$ is called the {\em chromatic index}  of $G$ and denoted by $\chi'(G)$.

It is  well-known that the chromatic index of a graph  $G$ is equal either to 
$\Delta$ or $\Delta+1$, where $\Delta$  denotes its maximum degree; we then say that $G$ is of {\em class 1 }or {\em class 2}, respectively.  

The  {\em palette} of a vertex 
$v \in V (G)$ with respect to a proper edge-coloring $f$ of $G$ is the set $P_f(v) =
\{f(e) \, : \, e \in E(G) \, {\rm and} \, e \,  {\rm is \, incident \, to} \, v\}$. 

If $f$ is a proper edge-coloring of $G$ and $X \subseteq V(G)$, then $p_f(X)$ 
is the number of distinct palettes of the vertices of $X$ with respect to $f$. 
The {\em palette index} of a graph $G$, denoted by  $\check s(G)$, is 
the minimal value of $p_f(V(G))$ taken over all proper edge-colorings of $G$.

The palette index was introduced by Horňák et al. in \cite{hornak}, where initial results on the palette index of cubic and complete graphs were provided.  
Subsequent studies have investigated the palette index for other regular graphs \cite{Bonvicini, italijani}.  
Further classes of graphs explored in relation to this invariant include trees \cite{Bonvicini, italijani}, complete bipartite graphs \cite{hornak2}, and Cartesian products.  
Specifically, \cite{Smbatyan} presents partial results on the palette index of the Cartesian product of a path and a cycle, while \cite{Casselgren} establishes the palette index of the Cartesian product of two paths.  

The palette index has also been analyzed in relation to the maximum and minimum degree of a graph \cite{italijani, Casselgren}.  
Additionally, it has found applications in modeling the self-assembly of DNA structures with branched junction molecules that possess flexible arms \cite{Bonvicini}.

In this paper, we build upon previous research on the palette index of the Cartesian product of two graphs, focusing on cases where one factor graph is a path or cycle. Additionally, we extend the study to families of Cartesian products where one of the factors is a regular or class 1 “nearly regular” graph. 

The next section provides necessary definitions and preliminary results used throughout the paper. Section 3 introduces the class of nearly regular graphs and presents results on the palette index of Cartesian products where one factor is either regular or a class 1 nearly regular graph. Notably, this section establishes that the palette index of a Cartesian product with a class 1 nearly regular graph as one factor is always 2. 

Section 4 continues the exploration of the palette index for Cartesian products involving a path or cycle as one factor. In particular, it includes a construction that produces an edge coloring with three palettes for the Cartesian product of two odd cycles. 

Finally, Section 5 concludes the paper by applying the results from earlier sections to determine the palette index of Cartesian products where one factor graph is either a cycle or a path and the other is a regular graph.

\section{Preliminaries}

The {\em Cartesian product} of graphs $G$ and $H$ is the graph $G \Box H$  with vertex set $V(G)\times V(H)$
and $(x_1,x_2)(y_1,y_2) \in E(G \Box H)$  whenever  $x_1y_1 \in E(G)$ and $x_2=y_2$,  or
 $x_2y_2 \in E(H)$ and $x_1=y_1$.  The Cartesian product is clearly commutative. 
 

Let $[n]$ and  $[n]_0$ denote the sets $\{1,2, \ldots, n\}$ and  $\{0,1, \ldots, n-1\}$, respectively.  
We will assume in the sequel that $V(P_n) = V(C_n) := [n]_0$ and 
$V(P_n \Box C_m) = V(C_n \Box C_m) := [n]_0 \times  [m]_0$ . 

Let $G$ be a graph. If $X \subset E(G)$, then the spanning subgraph of $G$ with the edge set 
$E(G) \setminus X$ is denoted by $G - X$. 

A {\em matching} in a graph \( G \) is a subset \( M \subseteq E(G) \) such that no two edges in \( M \) share a common vertex. Naturally, if \( f \) is a proper edge-coloring of \( G \), the set of edges assigned a specific color \( i \) under \( f \) constitutes a matching in \( G \). A matching \( M \) is called 
{\em perfect} if every vertex in \( V(G) \) is incident to exactly one edge from \( M \).


We begin with the following obvious observation.
\begin{obs1}  \label{observation}
If $G$ is a graph such that the degree of every vertex of $V(G)$ is from the set 
$\{d_1, d_2, \ldots, d_k \}$, then $\check s(G) \ge k$.
\end{obs1}
It is also not difficult to confirm the following result. 
\begin{prop1} \label{matching}
If  $M$ is a perfect matching of a nontrivial graph $G$,  then $\check s(G) \le \check s(G - M)$.
\end{prop1}

\begin{proof} 
Let $g: E(G-M) \rightarrow C$ be an edge coloring of $G - M$ with $\check s(G - M)$ palettes and
 let $c \not \in  C$.   It is easy to construct the  edge coloring $h$ of $G$,  where for every $e \in M$ we set 
 $h(e) := c$, while for every $e' \in E(G) \setminus M$ we set $h(e') := g(e')$.  
 Since for every  $v \in V(G)$ we have $P_h(v) = P_g(v) \cup \{ c  \}$, 
 it follows that  $\check s(G) \le \check s(G - M)$.
\end{proof} 

\begin{figure}[!ht]
	\centering
		\includegraphics[width=8cm]{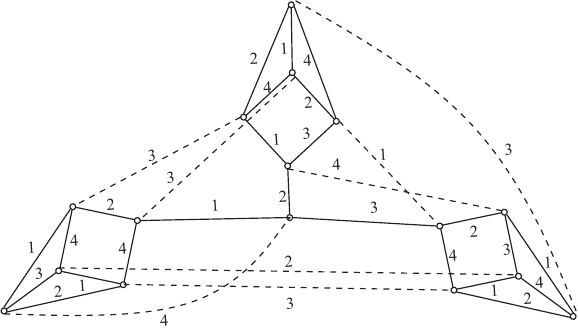}
	\caption{An edge coloring of a 4-regular graph $G$ with a perfect matching $M$}
\label{M}
\end{figure}

With respect to Proposition  \ref{matching}, it is worth noting that does not  necessarily  hold that $\check s(G) \ge \check s(G - M)$,  even if $G$ is a regular graph. 
Consider,  for instance,  an edge coloring of the 4-regular graph $G$ shown in Fig. \ref{M},  where a perfect matching $M$ is indicated by dashed lines.  In this case, we observe that 
$\check s(G) = 1$,  while $\check s(G - M) \not =  1$.
This discrepancy arises because
 $G-M$ is a cubic graph containing a vertex $v$ such that every edge incident to $v$ is a bridge of $G-M$. Consequently,   $\check s(G - M) = 4$  (see \cite[Proposition 3.3]{italijani}). 



The following two results are shown in \cite{hornak2}.
\begin{prop1} \label{hornak}
Let $G$ be a $r$-regular graph. Then $\chi'(G) = r$ if and only $\check s(G)=1$. 
\end{prop1}

\begin{prop1} \label{hornak2}
If $G$ is  a regular graph,  then $\check s(G) \not= 2$. 
\end{prop1}

As shown in \cite{mahmo}, the Cartesian product of two  graphs  is class 1 if at least one of the 
factors is class 1.  This result is more formaly stated in the next theorem.  
(Since the proof of the theorem is based on a construction that we will need in the sequel,  
we stated it explecitly  although an analogous aproach has been already used in \cite{mahmo}). 

\begin{thm} \label{mah}
Let $G$ and $H$ be graphs.  If $G$ is class 1 nontrivial graph,  then $G \Box H $ is class 1. 
\end{thm}

\begin{proof} 
Note first that $\Delta(G \Box H) = \Delta(G) + \Delta(H)$.  

Let $g: E(G) \rightarrow [\Delta(G)]$ be an edge coloring of $G$ and
$h: E(H) \rightarrow [\Delta(H)]$ 
(resp.  $h: E(H) \rightarrow [\Delta(H)+1]$)  an edge coloring of $H$
if $H$ is class 1 (resp. class 2). 

We construct an edge coloring $f$ with $\Delta(G)+ \Delta(H)$  color 
as follows. 

For every $x_2y_2 \in E(H)$ and every $z_1 \in V(G)$ we set $f((z_1,x_2)(z_1,y_2)) := h(x_2y_2)$.
If $H$ is class 1,  then  for every $x_1y_1 \in E(G)$ and every $z_2 \in V(H)$ we set 
$f((x_1,z_2)(y_1,z_2)) := g(x_1y_1) + \Delta(H)$.
Note that $f$ is clearly a proper edge coloring of $G \Box H$ with  $\Delta(G) + \Delta(H)$ colors.

If $H$ is class 2, then we obtain $f$ by choosing first an arbitrary color $c \in [\Delta(G)]$. 
Then for every $x_1y_1 \in E(G)$ and every $z_2 \in V(H)$ we set:  

   $f((x_1,z_2)(y_1,z_2)) := g(x_1y_1) + \Delta(H)$,   if  $g(x_1y_1) \not = c$,  

   $f((x_1,z_2)(y_1,z_2)) := c'$,   if  $g(x_1y_1) = c$,  where $c' \in [\Delta(H)+1] \setminus P_h(z_2)$.

Note that $c'$ always exists since $|P_h(z_2)| \le \Delta(H)$.  

We can see that $f$ is  a proper edge coloring of $G \Box H$ with  $\Delta(G) + \Delta(H)$ colors.
It follows that  $G \Box H$ is class 1.
\end{proof}

It will be needed in the sequel that, if in the proof of Theorem \ref{mah} we construct an edge coloring $f$ by choosing $c = \Delta(G)$, then $f: E(G \Box H) \rightarrow [\Delta(G) + \Delta(H)]$ is obtained.
 
 
 \section{Palette index of Cartesian products of regular and class 1  nearly regular graphs }

For graphs $G$ and $H$, it is shown in \cite{Smbatyan2} that the palete index  of $G \Box H$ is bounded above by the product of the palete indices of both factor graphs.
\begin{prop1} \label{bound}
If $G$ and $H$ are graphs, then $\check s( G \Box H) \le \check s(G) \check s(  H).  $
\end{prop1}

Considering Cartesian products of regular graphs, the following  observation from \cite{Bonvicini} is in noteworthy. 

\begin{prop1} \label{regular}
If $G$ is an $r$-regular graph 
then  $\check s(G) \le r+1$.
\end{prop1}

We will show in the sequel that for the Cartesian products of regular and some related graphs  the above upper bounds can be significantelly improved. 

Notice that the Cartesian product of two  regular graphs is clearly a regular graph. Thus, 
Theorem \ref{mah} and Proposition \ref{hornak} yield the following corollary. 

\begin{cor1} \label{c1}
Let $G$ and $H$ be regular graphs.  If $G$ is class 1 nontrivial graph,  then $\check s( G \Box H) = 1$. 
\end{cor1}


Let  $G'$ be a connected $r$-regular class 1 nontrivial graph and $G$ a spanning  subgraph of $G'$.
We say that $G$ is a {\em class 1  nearly regular graph (derived from $G'$)} or shortly {\em NRG}  if 
$G'$ admits a perfect matching $M$ such that  $G' - M$ is class 1  and $G = G' - X$,
where  $X \subset M$.

Let $g: E(G) \rightarrow [r]$ be an edge coloring of a graph $G$ and let $C^g_1, \ldots, C^g_r$ be the 
corresponding  color classes.   
Note that if $G$ is NRG derived from a $r$-regular graph $G'$,  then for every $j \in [r]$ there exists  
 an $r$-edge coloring $g$ of $G'$, such that $G = G' - X$, where  $X \subset C^g_j$. 

\begin{thm} \label{nrg}
Let $H$ be a connected regular graph.  If $G$ is NRG,  then $\check s(G \Box H) = 2$. 
\end{thm}

\begin{proof}
Since $G$ is not a regular graph,  we have $\check s(G \Box H) \ge 2$.  We will contruct a proper edge  coloring $f$ 
of $G \Box H$ with two distinct palettes.

Suppose that $H$ is a $r'$-regular graph, while 
$G$ is derived from  a class 1 $r$-regular graph  $G'$.  
It follows that there exists  an edge coloring $g: E(G') \rightarrow [r]$  of $G'$ 
 such that $G = G' - X$,  where  $X \subset C^g_r$.

Remind that by Corollary \ref{c1},  it holds that $G' \Box H$ is class 1.  

 Let $h: E(H) \rightarrow [r']$ (resp. $h: E(H) \rightarrow [r' + 1]$) be an edge coloring of $H$ 
 if $H$ is class 1 (resp. $H$ is class 2).   
 Since  $G'$ is class 1,  we can construct a 
 $(r + r')$-edge coloring  $f$ of $G' \Box H$  as we shown in the proof of Theorem \ref{mah}.  
 That is to say,  if $H$ is class 2, then for every $x_1y_1 \in E(G)$ and every $z_2 \in E(H)$ we 
 set: 
 
    $f((x_1,z_2)(y_1,z_2)) := g(x_1y_1) + r'+1$,   if  $g(x_1y_1) \not = r$,  

   $f((x_1,z_2)(y_1,z_2)) := c'$,   where $c' \in [r'+1] \setminus P_h(z_2)$,  if  $g(x_1y_1) = r$.

\noindent
It follows that  the palette of every vertex of  $G' \Box H$ with respect to $f$ equals $[r + r']$. 
Moreover,  $f$ restricted to  $G \Box H$ admits  two palettes: 

  $[r + r']$, for every vertex of degree $r + r'$; and 
 
  $[r + r' - 1]$,  for every vertex of degree $r + r' -1$,  i.e., a vertex incident to an edge of $X$ in $G' \Box H$.  
 
\noindent 
Since we showed that for every connected regular graph $H$ and  class 1  nearly regular graph $G$  
 we can always found an edge coloring of $G \Box H$ with  
two palettes,    it follows that $\check s(G \Box H) = 2$.
\end{proof}

If a connected regular graph $G$ is class 1,  then notice that $G-e$ is NRG for every  $e \in E(G)$. 
This observation provides the following corollary to Theorem \ref{nrg}. 

\begin{cor1} \label{c2}
Let $H$ and $G$ be connected regular graphs.  If $G$ is class 1 nontrivial graph and $e \in E(G)$,  then $\check s((G - e)  \Box H) = 2$. 
\end{cor1}

\begin{figure}[!ht]
	\centering
		\includegraphics[width=8cm]{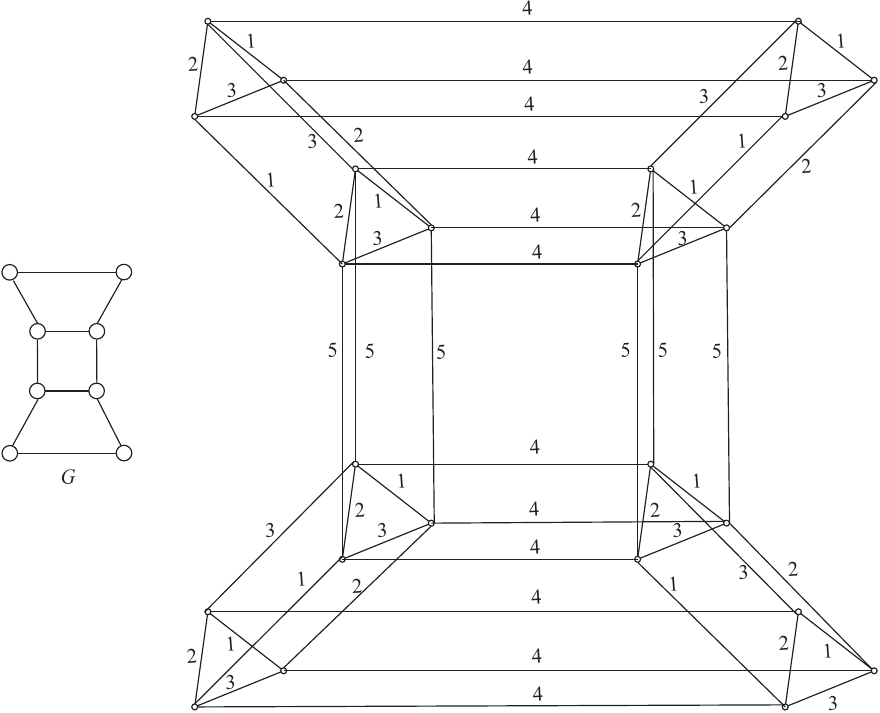}
	\caption{An edge coloring of  a subgraph of $G \Box C_3$ with 2 palettes, where $G$ is a NRG derived from $Q_3$}
\label{Q}
\end{figure}

A little more involved application of Theorem \ref{nrg} considers the well known class of hypecube graphs known as 
$r$-cubes. 
Remind that the vertex set of the {\em $r$-cube} $Q_r$ consists of all $r$-tuples $b_1\ldots b_r$,  $b_i \in \{0,  1 \}$.
Two vertices of  $Q_r$ are adjacent if corresponding $r$-tuples differ in precisely one coordinate.   
Note that $Q_1 = K_2$,  while  for $r\ge 2$ we have  $Q_r = Q_{r-1} \Box  K_2$. 
It is not difficult to see that $Q_r$ is class 1.

A subgraph $H$ of a graph $G$ is {\em isometric} if $d_H(u,v)=d_G(u,v)$ for any pair of vertices $u$ and $v$ from $H$.
Isometric subgraphs of hypercubes are called {\em partial cubes}.

Let $\alpha : V(G) \rightarrow V(Q_r)$ be an isometric embedding of $G$ into the $r$-cube,  i,e.,  for every 
$u, v \in V(G)$ we have $d_G(u,v)=d_{Q_r}(\alpha(u),\alpha(v))$.
We will denote the $i$-th coordinate of $\alpha$ with $\alpha_{i}$, i.e. $\alpha = (\alpha_{1}, \alpha_{2}, \ldots, \alpha_{r})$.

By the definition,  end-vertices of an arbitrary edge of $Q_r$ differ exactly in coordinate $i$ 
for some $i \in [r]$. 
Let $G$ be a partial cube with an isometric embedding $\alpha : V(G) \rightarrow V(Q_r)$.
The set of all edges $uv \in E(G)$ satisfing the condition $\alpha_i(u) \not =  \alpha_i(v)$
 is denoted by $E_i$,  more formaly: $E_i = \{uv \,  |  \, uv \in E(G),  \,  \alpha_i(u) \not =  \alpha_i(v) \}$.
  (These sets are known as classes of the equvalence relation $\Theta$,  see \cite{imkl-00} for the details.) 

Clearly,  the sets  $E_1, E_2,  \ldots, E_r$  partition the set of edges of $G$.  Moreover, 
the function $f: E(G) \rightarrow [r]$,  where $f(e)=i$ for every $e \in E_i$,  is a proper edge coloring  of 
$G$.  Thus, it is not difficult to see that the following result holds. 
 
 \begin{prop1} \label{kocka}
Let  $G$ be a partial cube  with an isometric embedding $\alpha : V(G) \rightarrow V(Q_r)$,  $r \ge 2$, 
and $H$ a regular graph.  If $X \subset E_i$,   $i\in [r]$,  where $E_i = \{uv \,  |  \, uv \in E(G),  \,  \alpha_i(u) \not =  \alpha_i(v) \}$,
then $$\check s((G - X) \Box H ) = 2.$$
 \end{prop1}
 
 As an example to Proposition \ref{kocka} consider an edge coloring of a subgraph of $Q_3 \Box C_3$ isomorphic 
 to  $G \Box C_3$,  where $G$ is  a NRG derived from $Q_3$.  The edge coloring $f$  of
 $G \Box C_3$ is constructed with respect to the proof of Theorem \ref{nrg}, where 
 $h: E(C_3) \rightarrow [3]$  and $g: E(Q_3) \rightarrow \{ 4, 5, 6\}$.
 
  \section{Palette index of Cartesian products with a path and cycle }

To establish the palette index for a path, note that $P_n$ admits a perfect matching if and only if $n$ is even.
\begin{prop1}
Let $n \ge 3$. Then
\begin{displaymath}
\check s(P_n) =
        \left \{ \begin{array}{llll}
             2, &  n  \; {\rm even}   \\
              3, &  n  \; {\rm odd}   \\
             \end{array}. \right.
\end{displaymath}
 \end{prop1}

The palette index for a cycle follows from Propositions \ref{hornak} and \ref{regular}.
\begin{prop1}
Let $n \ge 3$. Then
\begin{displaymath}
\check s(C_n) =
        \left \{ \begin{array}{llll}
             1, &  n  \; {\rm even}   \\
              3, &  n  \; {\rm odd}   \\
             \end{array}. \right.
\end{displaymath}
 \end{prop1}
 
 The palette index of the Cartesian product of two paths is presented  \cite{Casselgren}.
 
 \begin{thm} \label{tpp}
Let $s,t \ge 2$. Then 
\begin{displaymath}
\check s(P_s \Box P_t) =
        \left \{ \begin{array}{llll}
             1, &   s  = t = 2  \\
             2, &   { \rm min( } s, t) = 2,  \; {\rm max(} s,t) \ge 3   \\
             3, &   s, t \ge 3  \; {\rm and \; s\cdot t \;  is \;  even }   \\
             5, &   s, t \ge 3  \; {\rm and \; s\cdot t \;  is \;  odd}   \\
             \end{array}. \right.
\end{displaymath}
 \end{thm}

The palette index of the Cartesian product of a path and cycle is studied in \cite{Smbatyan} where 
 the following partial result is presented.
 
\begin{prop1} \label{ppc}
Let $s,t \ge 3$. 

(i) If $s$ and $t$ are both odd,  then $\check s(C_s \Box P_t) = 4$.

(ii) If $s$ is even,  then $\check s(C_s \Box P_t) = 2$.

 \end{prop1}

 \begin{figure}[!ht]
	\centering
		\includegraphics[width=12cm]{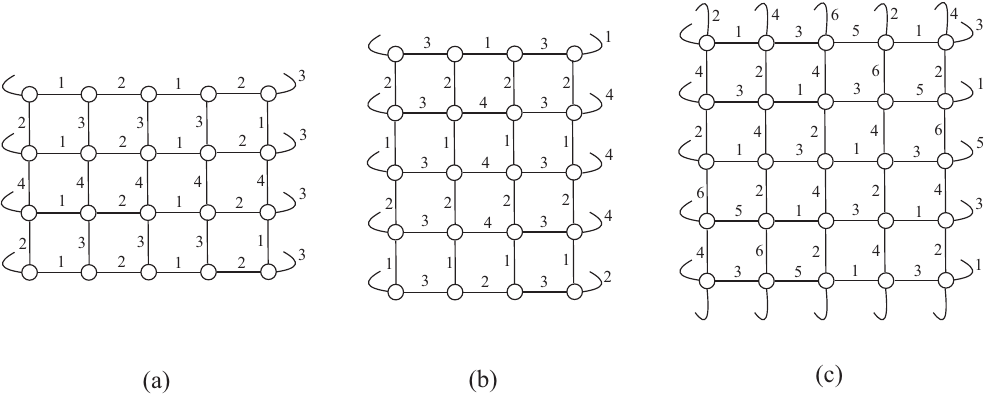}
	\caption{Edge colorings of:  (a) $C_5 \Box P_4$ with two palettes, 
	(b) $C_4 \Box P_5$ with two palettes
(c) $C_5  \Box  C_5$ with three palettes.}  
\label{pc}
\end{figure}

\begin{thm} \label{tpc}
Let $s,t \ge 3$. Then 
\begin{displaymath}
\check s(C_s \Box P_t) =
        \left \{ \begin{array}{llll}
             4, &   s  \; { \rm and \;  } t  \; {\rm are \; both \; odd}   \\
             2, & { \rm otherwise }     
             \end{array}. \right.
\end{displaymath}
\end{thm}

\begin{proof}
Clearly,  $\check s(C_s \Box P_t) \ge 2$ for every $s$ and $t$.  
With respect to Proposition \ref{ppc},  we have to confirm the result for every even $t$, i.e., 
to show that for every even integer $t$ and every $s$ there exists an edge coloring of $C_s \Box P_t$ 
with two palettes.  
Since the existence of a proper construction clearly follows from 
Corollary \ref{c2} (see an example in the left-hand side of Fig.  \ref{pc}),  
the proof is complete. 
\end{proof}

Notice also an example of an edge coloring of $C_4 \Box P_5$ with two palettes in 
 Fig.  \ref{pc}  (b) which can be generalized as we will show in Section 5.

In the rest of this section we consider the palette index of Cartesian products of two cycles.

Let  $G$ be a graph.  An {\em even cycle decomposition} of $G$ of size $k$ 
is a partition ${\cal E} = \{  \mathcal{E}_0,\mathcal{E}_1 \ldots,\mathcal{E}_{k-1} \}$
of the edge-set of $G$ such that the edges of $ \mathcal{E}_i$ compose disjoint even cycles of $G$.  
The studies presented in this section were inspired by  the results of Bonvicini and Mazzuoccolo \cite{Bonvicini2},  who showed 
that if a 4-regular graph $G$ admits palette index 3,  then
$G$ has an even cycle decomposition of size 3 or an even 2-factor.

Let $s \ge t \ge 3$  be  integers and $j \in [s]_0$,  $k \in [t]_0$,   
Let us define two types of vertical edges of $C_s \Box C_t$:

 $v_{j,k}^+ := (j,k)(j,(k+1) \, { \rm mod } \, t)$  (an ``ascending'' vertical edge),

 $v_{j,k}^- := (j,k)(j,(k-1) \,  {\rm mod } \, t)$  (a ``descending'' vertical edge);
 
 \noindent
 and a horizontal edge 
 
  $h_{j,k} := (j,k)((j+1) \, {\rm  mod } \, s,k)$.

The vertex $(j,k)$ is called the {\em initial vertex} of $v_{j,k}^+$, $v_{j,k}^-$ and $h_{j,k}$. 
The other end-vertex (i.e.  not initial) of a vertical or horizontal edge is called the {\em terminal vertex}.  
  

Let $\ell :=  \frac{s-t}{2}$  mod $t$ and $h :=  \lfloor \frac{s-t}{2t} \rfloor$.
 We will construct a partition of the edge set of $C_s \Box C_t$ based on the following sets:  
 $$Z_i^{s,t} = Z_{i,1}^{s,t}  \cup  Z_{i,2}^{s,t} \cup Z_{i,3}^{s,t}, \,i \in [t]_0$$ such that
 $$Z_{i,1}^{s,t} = \{ v_{j,i+j}^+,  h_{j,i+j+1} \, | \,  j \in [\ell ]_0 \},$$ 
$$ Z_{i,2}^{s,t} = \{ v_{j+\ell,i-j+\ell}^-,  h_{j+\ell,i-j+\ell-1} \, | \,  j \in [\ell]_0 \},$$
$$Z_{i,3}^{s,t} = \{ v_{j+2\ell,i-j}^-,  h_{j+2\ell,i-j-1} \, | \,  j \in [t(2h+1)]_0 \},$$
where  additions and substraction in the second coordinate are performed modulo $t$. 

Consider for example $Z_{2}^{13,5}$ depicted in Fig.  \ref{decomposition}, where the edges of this set are 
drawn with dashed lines.  Note that the initial vertex with the smallest first coordinate  of $Z_{2,1}^{13,5}$ 
is $(0,2)$,  while its  counterparts in $Z_{2,2}^{13,5}$ and $Z_{2,3}^{13,5}$ are 
$(4,1)$ and $(8,2)$, respectively. 

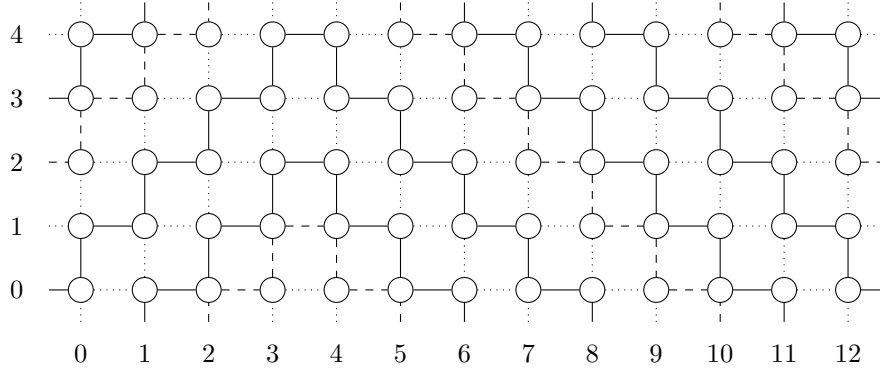
\begin{figure}[bt] 
\centering
\begin{tikzpicture}[scale=0.85]  \label{decomposition}.

  \foreach \a in {0, 1, ..., 12}
     \foreach \b in {0, 1, ..., 4 } {
         \node[black, circle, draw] (p{\a}d{\b}) at ({\a+1},{\b+1}) {};
     }  
    
\draw (p{0}d{0}) 
\foreach \a in {p{0}d{1},p{1}d{1},p{1}d{2},p{2}d{2},p{2}d{3},p{3}d{3},p{3}d{4},p{4}d{4}} 
{ -- (\a)};

\draw (p{4}d{4})
\foreach \a in {p{4}d{3},p{5}d{3},p{5}d{2},p{6}d{2},p{6}d{1},p{7}d{1},p{7}d{0},p{8}d{0}} 
{ -- (\a)};

\draw (p{8}d{4})
\foreach \a  in {p{9}d{4},p{9}d{3},p{10}d{3},p{10}d{2},p{11}d{2},p{11}d{1},p{12}d{1},p{12}d{0}} { -- (\a)};
\draw (0.5,1) -- (0.8,1); \draw (13.2,1) -- (13.5,1);  
\draw (9,0.5) -- (9,0.8); \draw (9,5.5) -- (9,5.2); 

\draw [dotted] (0.5,2) -- (0.8,2); \draw [dotted] (13.2,2) -- (13.5,2);  
\draw [dotted] (p{0}d{1}) 
 -- (p{0}d{2}) -- (p{1}d{2}) -- (p{1}d{3}) -- (p{2}d{3}) -- (p{2}d{4}) -- (p{3}d{4}); 
\draw [dotted] (4,0.5) -- (4,0.8); \draw [dotted] (4,5.5) -- (4,5.2); 
\draw [dotted] (5,0.5) -- (5,0.8); \draw [dotted] (5,5.5) -- (5,5.2); 
\draw [dotted] (p{3}d{0}) -- (p{4}d{0}); 
\draw [dotted] (p{4}d{4}) 
 -- (p{5}d{4})  -- (p{5}d{3})  -- (p{6}d{3}) 
 -- (p{6}d{2})  -- (p{7}d{2})  -- (p{7}d{1}) -- (p{8}d{1}); 
\draw [dotted] (p{8}d{1}) -- (p{8}d{0})  -- (p{9}d{0}); 
\draw [dotted] (10,0.5) -- (10,0.8); \draw [dotted] (10,5.5) -- (10,5.2); 
\draw [dotted] (p{9}d{4})
\foreach \a  in {p{10}d{4},p{10}d{3},p{11}d{3},p{11}d{2},p{12}d{2},p{12}d{1}} { -- (\a)};

\draw [dashed] (0.5,3) -- (0.8,3); \draw [dashed] (13.2,3) -- (13.5,3);  
\draw [dashed] (p{0}d{2})
\foreach \a  in {p{0}d{3},p{1}d{3},p{1}d{4},p{2}d{4}} { -- (\a)};
\draw [dashed] (3,0.5) -- (3,0.8); \draw [dashed] (3,5.5) -- (3,5.2); 
\draw [dashed] (6,0.5) -- (6,0.8); \draw [dashed] (6,5.5) -- (6,5.2); 
\draw [dashed] (p{2}d{0})
\foreach \a  in {p{3}d{0},p{3}d{1},p{4}d{1},p{4}d{0},p{5}d{0}} { -- (\a)};
\draw [dashed] (p{5}d{4})
\foreach \a  in {p{6}d{4},p{6}d{3},p{7}d{3},p{7}d{2},p{8}d{2},
p{8}d{1},p{9}d{1},p{9}d{0},p{10}d{0}} { -- (\a)};
\draw [dashed] (11,0.5) -- (11,0.8); \draw [dashed] (11,5.5) -- (11,5.2); 
\draw [dashed] (p{10}d{4})
\foreach \a in {p{11}d{4},p{11}d{3},p{12}d{3}, p{12}d{2}} { -- (\a)};

\draw (0.5,4) -- (0.8,4); \draw  (13.2,4) -- (13.5,4);  
\draw  (p{0}d{3})
\foreach \a  in {p{0}d{4},p{1}d{4}} { -- (\a)};
\draw (2,0.5) -- (2,0.8); \draw  (2,5.5) -- (2,5.2); 
\draw  (7,0.5) -- (7,0.8); \draw  (7,5.5) -- (7,5.2); 
\draw (p{1}d{0})
\foreach \a  in {p{2}d{0},p{2}d{1},p{3}d{1},p{3}d{2},p{4}d{2},
p{4}d{1},p{5}d{1},p{5}d{0},p{6}d{0}} { -- (\a)};
\draw (p{6}d{4})
\foreach \a  in {p{7}d{4},p{7}d{3},p{8}d{3},p{8}d{2},p{9}d{2},p{9}d{1},p{10}d{1},p{10}d{0},p{11}d{0}} { -- (\a)};
\draw  (12,0.5) -- (12,0.8); \draw  (12,5.5) -- (12,5.2); 
\draw (p{11}d{4})
\foreach \a  in {p{12}d{4},p{12}d{3}} { -- (\a)};

\draw [dotted] (0.5,5) -- (0.8,5); \draw [dotted] (13.2,5) -- (13.5,5);  
\draw [dotted] (1,0.5) -- (1,0.8); \draw [dotted] (1,5.5) -- (1,5.2); 
\draw [dotted] (p{7}d{4})
\foreach \a  in {p{8}d{4},p{8}d{3},p{9}d{3},p{9}d{2},p{10}d{2},p{10}d{1},
p{11}d{1},p{11}d{0},p{12}d{0}} { -- (\a)};
\draw [dotted] (13,0.5) -- (13,0.8); \draw [dotted] (13,5.5) -- (13,5.2); 
\draw [dotted] (p{0}d{0})
\foreach \a  in {p{1}d{0},p{1}d{1},p{2}d{1},p{2}d{2},p{3}d{2},
p{3}d{3}, p{4}d{3}, p{4}d{2},p{5}d{2},p{5}d{1},p{6}d{1},p{6}d{0},p{7}d{0}} 
{ -- (\a)};
\draw [dotted] (8,0.5) -- (8,0.8); \draw [dotted] (8,5.5) -- (8,5.2); 

\draw (0,1) node {0}; \draw (0,2) node {1}; \draw (0,3) node {2}; \draw (0,4) node {3}; \draw (0,5) node {4};

\draw (1,0) node {0}; \draw (2,0) node {1}; \draw (3,0) node {2}; \draw (4,0) node {3};
\draw (5,0) node {4}; \draw (6,0) node {5}; \draw (7,0) node {6}; \draw (8,0) node {7};
\draw (9,0) node {8}; \draw (10,0) node {9}; \draw (11,0) node {10}; \draw (12,0) node {11};
\draw (13,0) node {12};

\end{tikzpicture}

\caption{An even cycle decomposition of  $C_{13} \Box C_{5}$}
\end{figure}

\begin{prop1} \label{part}
If $s \ge t \ge 3$ are odd integers,  then 
$\{ Z_0^{s,t}, Z_1^{s,t}, \ldots,   Z_{t-1}^{s,t}\} $  partition the set of edges of
$C_s \Box C_t$. 
\end{prop1}

\begin{proof}
Note that  the first coordinates of vertices from the set $Z_i^{s,t}$ are pairwise distinct. 

We first show that for every $i \not = k$ we have   $Z_i^{s,t}\cap  Z_k^{s,t} = \emptyset$.  To confirm this,  notice that
for every $j$ and every $i,k$,  $i \not = k$, we have  $v_{j,i+j}^+ \not = v_{j,k+j}^+$,  
$v_{j+\ell,i-j+\ell}^- \not = v_{j+\ell,k-j+\ell}^-$,  $v_{j+2\ell,i-j}^- \not = v_{j+2\ell,k-j}^-$, 
$h_{j,i+j+1} \not = h_{j,k+j+1}$, 
$h_{j+\ell,i-j+\ell-1} \not = h_{j+\ell,k-j+\ell-1}$,  $h_{j+2\ell,i-j-1} \not = h_{j+2\ell,k-j-1}$,  i.e., 
the initial vertices of $Z_i^{s,t}$ and $Z_k^{s,t}$ do not coincide.
It follows that $Z_i^{s,t}\cap  Z_k^{s,t} = \emptyset$.

Since $|Z_i^{s,t}| = 4 \ell + 2t(2h+1)$, $\ell =  \frac{s-t}{2}$  mod $t$ and $h =  \lfloor \frac{s-t}{2t} \rfloor$, 
we obtain $|Z_i^{s,t}| = 2s$. It follows that 
$$|\cup_{i=0}^{t-1} Z_i^{s,t}| = \sum_{i=0}^{t-1} |Z_i^{s,t}| = 2st = |E(C_s \Box C_t)|.$$
This assertion completes the proof.
\end{proof}

\begin{prop1} \label{even}
Let $s \ge t \ge 3$ be odd integers. If $\mathcal{E}_j = \cup_{i \in [t]_0,  i \equiv j  \, \rm{(mod } \, 3)} Z_i^{s,t}$,
  then ${\cal E} = \{  \mathcal{E}_0,\mathcal{E}_1, \mathcal{E}_{2} \}$
  is an even cycle decomposition 
  of $C_s \Box C_t$. 
\end{prop1}

\begin{proof}
Note first that in the set $Z_i^{s,t}$ the following pairs of edges are incident:
\begin{itemize}
\item 
 for every $j \in  [\ell]_0$:   $v_{j,i+j}^+$ and  $h_{j,i+j+1}$ (edges in $Z^{s,t}_{i,1}$); 
 $v_{j+\ell,i-j+\ell}^-$ and $h_{j+\ell,i-j+\ell-1}$ (edges  in $Z^{s,t}_{i,2}$); 

\item for every $j \in [t(2h+1)]_0$:  $v_{j+2\ell,i-j}^-$ and $h_{j+2\ell,i-j-1}$ (edges  in $Z^{s,t}_{i,3}$); 

\item 
 for every $j \in  [\ell-1]_0$:  $h_{j,i+j+1}$ and $v_{j+1,i+j+1}^+$ (edges  in $Z^{s,t}_{i,1}$);  
$h_{j+\ell,i-j+\ell-1}$ and $v_{j+\ell +1,i-j+\ell+1l}^-$ (edges  in $Z^{s,t}_{i,2}$); 
\item for every $j \in [t(2h+1)]_0$:    $h_{j+2\ell,i-j-1}$ and $v_{j+2\ell +1,i-j+1}^-$ (edges  in $Z^{s,t}_{i,3}$).  
 \end{itemize}
  
Moreover,   $h_{\ell-1,i+\ell}$ is incident to $v_{\ell,i+\ell}^- $,  
$h_{2\ell-1,i}$ is incident to $v_{2\ell,i}^- $,  and
$h_{s-1,i} $ is incident to $v_{0,i}^-$.  

It follows that edges of $Z_i^{s,t}$ compose a circle in  $C_s \Box C_t$. 
As we can see in the proof of Proposition \ref{part}, this cycle is of legth $2s$.  
 Furthermore,  if $i  \equiv k \equiv j$ (mod 3)  and $i \not = k$,  then  every initial vertex of an edge of 
$Z_i^{s,t}$ and  every initial vertex of an edge of $Z_k^{s,t}$ are at distance at least 2.   Hence, 
the edges of $\mathcal{E}_j$ compose  disjoint even cycles of $C_s \Box C_t$. 
From Proposition  \ref{part} now it follows that ${\cal E} = \{  \mathcal{E}_0,\mathcal{E}_1, \mathcal{E}_{2} \}$
  is an even cycle decomposition of $C_s \Box C_t$. 
\end{proof}

An even cycle decomposition of  $C_{13} \Box C_{5}$ is depicted in Fig.  \ref{decomposition}. 

\begin{thm} \label{CnCm}
Let $s \ge t \ge 3$.  Then
\begin{displaymath}
\check s(C_s \Box C_t) =
        \left \{ \begin{array}{llll}
             1, &   s  \; { \rm or \;  } t  \; {\rm is \;  even}   \\
             3, &   s  \; { \rm and \;  } t  \; {\rm are \; both \; odd}   \\
             \end{array}. \right.
\end{displaymath}
\end{thm}

\begin{proof}
For $s$ or $t$ even,  the theorem follows from Corollary \ref{c1} since a cycle of even length is clearly class 1.  

If $s$ and $t$ are both odd, then $C_s \Box C_t$ is  class 2 and by 
Proposition \ref{hornak2} we have $\check s(C_s \Box C_t) \ge 3$. 
As shown in \cite{Bonvicini2},  
the existence of an even cycle decomposition of size 3  in  $C_n \Box C_m$ implies 
that $C_n \Box C_m$ admits the palette index 3.  

It is shown in Proposition \ref{even} that ${\cal E} = \{  \mathcal{E}_0,\mathcal{E}_1, \mathcal{E}_{2} \}$
is an even cycle decomposition   of $C_s \Box C_t$ of size 3, where   
$\mathcal{E}_j = \cup_{i \in [t]_0,  i \equiv j  \, \rm{(mod } \, 3)} Z_i^{s,t}$.  
We now construct the  edge coloring $c : E(C_s \Box C_t) \rightarrow [6]$ as follows:

\begin{displaymath}
 c(e) =
        \left \{ \begin{array}{llll}
             2j + 1, &   e  \; { \rm is \;a \; horizontal \; edge \;in}   \;     \mathcal{E}_j  \\
             2j + 2, &   e  \; { \rm is \;a \; vertical \; edge\; in}   \;     \mathcal{E}_j   \\
             \end{array}. \right.
\end{displaymath}
By the definition of the set  $Z_i^{s,t}$,  every terminal vertex of an ascending vertical edge
of $Z_i^{s,t}$ equals the inital vertex of an ascending vertical edge of $Z_{i+1}^{s,t}$, 
while every inital  vertex of a descending  vertical edge
of $Z_i^{s,t}$ corresponds to the terminal vertex of a descending vertical edge of $Z_{i+1}^{s,t}$ (addition modulo 3). 

It follows that the palette of a vertex of $C_s  \Box  C_t$
with respect to $c$ is either $\{1,2,3,4\}$,  $\{1,2,5,6\}$ or $\{3,4,5,6\}$.  This assertion 
completes the proof.
\end{proof}

An example of a proper edge coloring with three palettes of  $C_5  \Box  C_5$ is depicted  in Fig. \ref{pc}.


\section{Paths, cycles and regular graphs}
In this section, we show that general upper bounds on the palette index of a Cartesian product can be significantly improved when one factor graph is a cycle or path and the other is a regular graph.

 \begin{thm} \label{CnG}
Let $G$ be a nontrivial regular graph. If $s \ge 3$,  then $$ {\it  ( i)}  \;  \check s(C_s \Box G)  \le \check s( G) + 2.$$
Moreover,   

 (ii) if $s$ is even or $G$ is class 1, then $\check s(C_s \Box G) = 1$,

 (iii) if $G$ is a class  2 cubic graph with a perfect matching and $s$ is odd, then $\check s(C_s \Box G) \in  \{ 1, 3 \}$.
\end{thm}

\begin{proof}
If  $s$ is even or $G$ is class 1, then one of the factor graphs is class 1 and (ii)  follows from Corrolary  \ref{c1}. 

To prove (i),   suppose then that $s$ is odd and $G$ a class 2 $r$-regular 
graph.  Let  $g: E(G) \rightarrow [r+1]$ be a proper edge coloring of $G$ and let 
$h$ be a proper edge coloring of $G$ with $\check s(G)$ palettes such that $h(v) \not \in \{ r + 2, r +3  \}$  for every $v \in V(G)$.   
We will construct  a proper edge coloring $f$ of $C_s \Box G$ for $u,v \in V(G)$  and $i,j \in [s]_0$
as follows:

\begin{displaymath}
 f((u,i)(v,j)) =
        \left \{ \begin{array}{llll}
             g(uv),   &   uv \in   E(G)  \;   { \rm and} \;   i = j \not = s - 1     \\
             h(uv),   &   uv \in   E(G)  \;   { \rm and} \;   i = j = s - 1     \\
             c,  &  u = v,  \;  i \in  [s-2]_0 \; { \rm is \;even \;and} \; j = i + 1, \; \\ &  { \rm where} \; c \in [r+1] \setminus  P_g(v)  \\
             r+2,  &  u = v,  \;  i \in  [s-2] \; { \rm is \;odd \;and} \; j = i + 1  \\
             r+3,  &  u = v,  \;  i  = s-1 \; { \rm and} \; j = 0  
            \end{array}. \right.
\end{displaymath}
Since we can see that 

 - for every $i \in [s-2]$ and every $v \in V(G)$ we have $P_f((v,i)) = [r+2]$,

 - for every $v \in V(G)$ we have $P_f((v,0)) = [r+1] \cup \{ r + 3\}$,

 - for every $v \in V(G)$ we have $P_f((v,s-1)) = P_h(v)  \cup \{ r+2,  r + 3\}$, 
 
\noindent
case (i) is settled.

To prove (iii), let $G$ denote a  class  2 cubic graph with a perfect matching $M$ and $F_M$ the corresponding 
1-factor of  $G$, i.e.,  a 1-regular spanning subgraph of $G$. 
Note that $E(C_s  \Box F_M)$ is 
a  perfect matching of $C_s  \Box G$ and  $C_s  \Box G - E(C_s  \Box F_M) = C_s \Box (G-M)$. 

Since $G$ is class  2,  $G$ is not bipartite. 
Moreover,  since $G-M$ is 2-regular,  its connected 
component are cycles with at least one of them having  odd length.  
Thus, $C_s \Box (G-M)$ is a graph whose 
 connected components are  Cartesian products of two cycles. 
Remind from Theorem \ref{CnCm}  that the palette index of the Cartesian product of two cycles is either 1 (at least one of them is of even length) or   3
(both of them are of odd length).  
Therefore,  $\check s(C_s \Box (G-M)) = 3$. 
 By Proposition \ref{matching}, we have 
 $\check s(C_s \Box G) \le \check s(C_s \Box (G-M))  = 3$.  Finally,   Proposition \ref{hornak2} completes the proof.
\end{proof}

\begin{figure}[!ht]
	\centering
		\includegraphics[width=10cm]{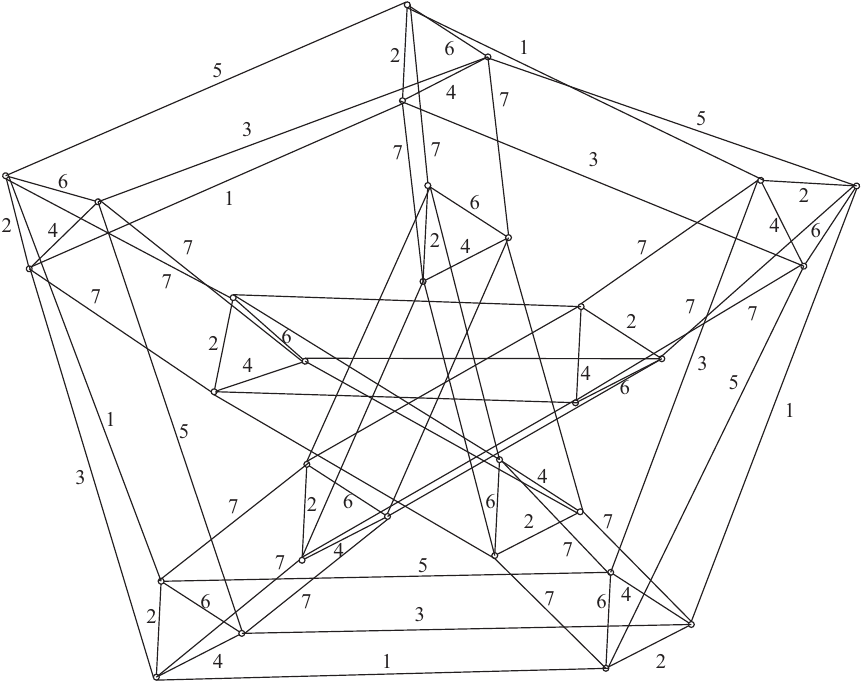}
	\caption{An (incomplete) edge coloring of the Cartesian product of the Petersen graph and triangle with three palettes}
\label{pet}
\end{figure}


As an example of edge colorings provided by Theorem \ref{CnG} (ii), observe Fig. \ref{pet}, where an edge coloring of the Cartesian product of the Petersen graph and the triangle with three palettes is partially depicted. (For clarity, the labeling of the "inner" product of $C_3$ and $C_5$ with colors 1, 3, and 5 is omitted.)
 \begin{thm} \label{PnG}
Let $G$ be a nontrivial regular graph. If $s \ge 3$,  then $$ {\it  ( i)}  \;  \check s(P_s \Box G)  \le \check s( G) + 2.$$
Moreover,   

 (ii) if $s$ even or $G$ is class 1, then $\check s(P_s \Box G) = 2$,

 (iii)  if $G$ is a class  2 cubic graph with a perfect matching and $s$ is odd, then $\check s(P_s \Box G)  \in    \{ 2,  3, 4 \}$.
\end{thm}

\begin{proof}
Clearly, $\check s(P_s \Box G) \ge  2$.

If  $s$ is even,  then $P_s$ is a NRG and  (ii) follows from Theorem \ref{nrg}. 
If $G$ is class 1 and $s$ is odd,    we construct  a suitable edge coloring of $P_s \Box G$ 
in the sequel.   

Let $g$, $g'$ and $g''$  be proper edge colorings of $G$ with exactly one palette,  such that for every $v \in V(G)$
we have  $1,2  \not \in P_g(v))$,
$P_{g'}(v) = \{2 \} \cup (P_g(v) \setminus \{ c \})$ and  $P_{g''}(v) = \{1 \} \cup (P_g(v) \setminus \{ c \})$,  where  $c$ is an 
arbitrary color of $g$.  We  contruct a proper edge  coloring $f$  of $P_s \Box G$ with two distinct palettes as follows:

 - for every odd $i \in [s-2]$ and every $v \in V(G)$ we set $f((v,i)(v,i-1)) = 1$ and $f((v,i)(v,i+1)) = 2$, 

 - for every $i \in [s-2]$ and every $uv \in E(G)$  we set $f((u,i)(v,i)) = g(uv)$,

 - for every  $v \in V(G)$  we set $f((u,0)(v,0)) = g'(uv)$ and $f((u,s-1)(v,s-1)) = g''(uv)$.

\noindent
For example consider an edge coloring of $P_5 \Box C_4$ depicted in Fig. \ref{pc} (b). 
 Note that $g$, $g'$, and $g"$  are edge colorings of 
$C_4$ with one palette, where for every $v \in V(C_4)$ we have $P_{g}(v) = \{ 3, 4 \}$, $P_{g'}(v) = \{ 2, 3 \}$ and $P_{g"}(v) = \{ 1, 3 \}$.

Since we can see that  for every $i \in [s-2]$ and every $v \in V(G)$ we have $P_f((v,i)) = P_g(v) \cup \{ 1, 2 \}  $,
while for every $v \in V(G)$ we have $P_f((v,0)) = P_f((v, s-1)) = (P_g(v) \setminus \{ c \})  \cup \{ 1, 2 \} $,
case (ii) is settled.  

To prove (i),    consider  again a proper edge coloring $f$ of $C_s \Box G$ contructed in the proof of 
Theorem \ref{CnG}, where $s$ is odd and $G$ a class 2 $r$-regular graph. 
It is not difficult to see that $f$ restricted to $P_s \Box G$ 
for every $i \in [s-2]$ and every $v \in V(G)$ implies $P_f((v,i)) = [r+2]$, while 
 for every $v \in V(G)$ we have $P_f((v,0)) = [r+1]$ and $P_f((v,s-1)) = P_h((v))  \cup \{ r + 2\}$.
This argument settles the proof of case (i).

To prove (iii), let $G$ be a cubic graph with a perfect matching $M$, and let $F_M$ be the corresponding 1-factor of $G$. Analogously to the proof of Theorem \ref{CnG} (iii), we notice that $P_s \Box (G - M)$ is a graph whose connected components are Cartesian products of a path and a cycle, such that at least one of the factors is induced on an odd number of vertices. 

By Theorem \ref{tpc}, the palette index of the Cartesian product of a path and a cycle is either 2 (if at least one of the factors has an even number of vertices) or 4 (if both factors have an odd number of vertices). Thus, we obtain that $\check s(P_s \Box (G - M)) = 4$. 

By Proposition \ref{matching}, we have $\check s(P_s \Box G) \le \check s(P_s \Box (G - M)) = 4$. It follows that $\check s(P_s \Box G) \in \{ 2, 3, 4 \}$, and the proof is complete.
\end{proof}

Note that the bound provided by Theorem \ref{CnG} (i) improves the bounds from Propositions \ref{bound} and \ref{regular}. For example, let $G$ be isomorphic to $K_7$. From \cite{hornak}, we know that $\check{s}(K_7) = 3$, which implies $\check{s}(C_s \Box K_7) \leq 5$. In contrast, Propositions \ref{bound} and \ref{regular} provide the weaker bound $\check{s}(C_s \Box K_7) \leq 9$.
However, it remains unknown whether the bounds given by Theorem \ref{CnG} (i) and Theorem \ref{PnG} (i) are sharp. 

Regarding the lower bound on $\check{s}(G \Box H)$, it is worth noting that, in general, it does not depend on $\check{s}(G)$ and $\check{s}(H)$. Specifically, if $G$ and $H$ are regular graphs that both contain a perfect matching, then $\check{s}(G \Box H) = 1$, even if $\check{s}(G) > 1$ and $\check{s}(H) > 1$ (see \cite[Theorem 2.2]{mohar}).
\section*{Funding}
This work was supported by the Slovenian Research Agency under the grant P1-0297.


\end{document}